\documentclass[a4paper,12pt]{amsart}

\usepackage{eucal,fullpage,times,amsmath,amsthm,amssymb,mathrsfs,stmaryrd,color,enumerate,accents}
\usepackage[all]{xy}
\usepackage{url}
\usepackage[leqno]{amsmath}
\usepackage{amsthm}
\usepackage{amsfonts}
\usepackage{amssymb}
\usepackage{eucal}
\usepackage[all]{xy}
\CompileMatrices

\usepackage{mathrsfs}
\usepackage[notcite,notref]{showkeys}

\usepackage{amsmath,amssymb,mathrsfs,xy,amsthm,bm,url}
\input{xy}
\xyoption{all}

\usepackage{xspace}
 \usepackage[ansinew]{inputenc}
 \usepackage[T1]{fontenc}
 \usepackage{hyperref}

\usepackage{hyperref}

\theoremstyle{definition}
  \newtheorem{definition}[subsection]{Definition}
  \newtheorem{notation}[subsection]{Notation}
  \newtheorem{definition-proposition}[subsection]{Definition-Proposition}
  \newtheorem{example}[subsection]{Example}
  \newtheorem{remark}[subsection]{Remark}

\theoremstyle{theorem}
  \newtheorem{theorem}[subsection]{Theorem}
  
  \newtheorem*{proposition*}{Proposition}
  \newtheorem{lemma}[subsection]{Lemma}
  \newtheorem{corollary}[subsection]{Corollary}
  \newtheorem{conjecture}[subsection]{Conjecture}

\newcommand{\maxx}{{\mathrm{max}}}

\newcommand{\Cbb}{\mathbb{C}}

\newcommand{\Nbb}{\mathbb{N}}
\newcommand{\Qbb}{\mathbb{Q}}
\newcommand{\Rbb}{\mathbb{R}}
\newcommand{\Sbb}{\mathbb{S}}
\newcommand{\Zbb}{\mathbb{Z}}

\newcommand{\Cbf}{\mathbf{C}}
\newcommand{\Hbf}{\mathbf{H}}
\newcommand{\Gbf}{\mathbf{G}}
\newcommand{\ibf}{\mathbf{i}}

\newcommand{\Nbf}{\mathbf{N}}
\newcommand{\Pbf}{\mathbf{P}}

\newcommand{\Tbf}{\mathbf{T}}
\newcommand{\Ubf}{\mathbf{U}}
\newcommand{\Vbf}{\mathbf{V}}
\newcommand{\Wbf}{\mathbf{W}}

\newcommand{\Gfrak}{{\mathfrak{G}}}

\newcommand{\Xfrak}{\mathfrak{X}}
\newcommand{\Yfrak}{\mathfrak{Y}}

\newcommand{\Hscr}{\mathscr{H}}

\newcommand{\Pscr}{\mathscr{P}}
\newcommand{\Sscr}{\mathscr{S}}

\newcommand{\Ccal}{\mathcal{C}}

\newcommand{\Rcal}{\mathcal{R}}

\newcommand{\Ascr}{{\mathscr{A}}}
\newcommand{\Mscr}{{\mathscr{M}}}
\newcommand{\Yscr}{{\mathscr{Y}}}

\newcommand{\der}{\mathrm{der}}

\newcommand{\ra}{\rightarrow}
\newcommand{\lra}{\longrightarrow}

\newcommand{\mono}{\hookrightarrow}
\newcommand{\epim}{{\twoheadrightarrow}}	
\newcommand{\isom}{\cong}
\newcommand{\limproj}{\varprojlim}

\newcommand{\GSp}{\mathrm{GSp}}

\newcommand{\Res}{\mathrm{Res}}

\newcommand{\GL}{{\mathbf{GL}}}

\newcommand{\bsh}{\backslash}
\newcommand{\inv}{{-1}}
\newcommand{\ot}{\overset}

\newcommand{\wrt}{{with\ respect\ to}\xspace}
\newcommand{\ifof}{{if\ and\ only\ if}\xspace}
\newcommand{\cosg}{{compact\ open\ subgroup}\xspace}
\newcommand{\cosgs}{{compact\ open\ subgroups}\xspace}

\newcommand{\Qac}{\bar{\mathbb{Q}}}
\newcommand{\Gal}{\mathrm{Gal}}

\newcommand{\supp}{\mathrm{supp}}

\newcommand{\mult}{{\mathbb{G}_\mathrm{m}}}
\newcommand{\adele}{{\hat{\mathbb{Q}}}}

\newcommand{\Qbbp}{{\mathbb{Q}_p}}

\newcommand{\TW}{{{(\mathbf{T},w)}}}

\title{Bounded equidistributionof special subvarieties I}
\begin{document}

\maketitle

\tableofcontents

\section*{Introduction to the main results}

In this paper, we study the equidistribution of certain families of special subvarieties in a general mixed Shimura variety, and  the Andr\'e-Oort conjecture for these varieties, as a generalization of \cite{clozel ullmo} and \cite{ullmo yafaev}.

The notion of mixed Shimura varieties is developed in \cite{brylinski} and \cite{pink thesis}. They serve as moduli spaces of mixed Hodge structures, and often arise as boundary components of toroidal compactification of pure Shimura varieties. Among mixed Shimura varieties there are Kuga varieties, cf. \cite{chen kuga}, which are certain "universal" abelian scheme over Shimura varieties, and in general, a mixed Shimura variety can be realized as a torus bundle over some Kuga varieties (namely a torsor whose structure group is a torus). Similar to the pure case, we have the notion of special subvarieties in mixed Shimura varieties.

Y. Andr\'e and F. Oort conjectured that the Zariski closure of a sequence of special subvarieties in a given pure Shimura variety remains a finite union of special subvarieties, cf. \cite{andre note}. R. Pink proposed a generalization of this conjecture for mixed Shimura varieties by combining it with the Manin-Mumford conjecture and the Mordell-Lang conjecture for abelian varieties:

\begin{conjecture}[Pink, \cite{pink conjecture}]

Let $M$ be a mixed Shimura variety, and let $(M_n)_n$ be a sequence of special subvarieties. Then the Zariski closure of $\bigcup_nM_n$ is a finite union of special subvarieties.

\end{conjecture}

Remarkable progresses have been made for the Andr\'e-Oort conjecture in the pure case, cf. \cite{klingler yafaev}, \cite{scanlon bourbaki}, \cite{yafaev bordeaux}. In this paper, we generalize part of the  strategy of Klingler-Ullmo-Yafaev to the mixed case. The main ingredients of the strategy in the pure case can be expressed as the following ''ergodic-Galois'' alternative:

\begin{itemize}

\item equidistribution of special subvarieties with bounded Galois orbits (using ergodic theory), cf. \cite{clozel ullmo} and \cite{ullmo yafaev};

\item for a sequence of special subvarieties $(M_n)$ of unbounded Galois orbits, one can construct a new sequence of special subvarieties  $(M_n')$ such that $\bigcup_nM_n$ have the same Zariski closure as $\bigcup_nM_n'$ does, and that $\dim M_n<\dim M_n'$ for $n$ large enough.

\end{itemize} Note that both ingredients involves estimations that rely on the GRH (Generalized Riemann Hypothesis).

Our main result is the following:

\begin{theorem} Let $M$ be a mixed Shimura variety, and let $M_n$ be a bounded sequence of special subvarieties in $M$.  Then the Zariski closure of $\bigcup_n M_n$ is a finite union of special subvarieties.
\end{theorem}

Here bounded sequences means a sequence of special subvarieties that are $\TW$-special for finitely many pairs $\TW$, where $\TW$-special subvarieties is the analogue of $\Tbf$-special subvarieties in the mixed case, cf. the pure case in \cite{ullmo yafaev}.


We briefly explain the main idea of the paper. A mixed Shimura datum in the sense of \cite{pink thesis} is of the form $(\Pbf,\Ubf,Y)$ with $\Pbf$ a connected linear $\Qbb$-group, with a Levi decomposition $\Pbf=\Wbf\rtimes\Gbf$, $\Ubf$ a normal unipotent $\Qbb$-subgroup of $\Pbf$ central in $\Wbf$, and $Y$ a complex manifold homogeneous under $\Ubf(\Cbb)\Pbf(\Rbb)$ subject to some algebraic constraints. From mixed Shimura data we can define mixed Shimura varieties and their special subvarieties just like the case of pure Shimura varieties.  When $\Ubf=0$, we get Kuga data and Kuga varieties, cf. \cite{chen kuga}.

Similar to the pure case studied in \cite{clozel ullmo} and \cite{ullmo yafaev}, we first consider the Andr\'e-Oort conjecture for sequences of $\TW$-special subvarieties in a mixed Shimura variety $M$ defined by $(\Pbf,Y)$. Here $\Tbf$ is a $\Qbb$-torus in $\Gbf$ and $w$ an element of $\Wbf(\Qbb)$. Using the language of \cite{chen kuga} 2.10 etc, in a Kuga variety $M=\Gamma_\Vbf\rtimes\Gamma_\Gbf\bsh Y^+$ defined by $(\Pbf,Y)=\Vbf\rtimes(\Gbf,X)$,  $\TW$-special subvarieties are defined by subdata of the form $\Vbf'\rtimes(w\Gbf'w^\inv,wX')$, and they are obtained  from diagrams of the form $$\xymatrix{ M'\ar[r]^\subset & M_{S'}\ar[r]^\subset \ar[d]^\pi & M \ar[d]^\pi\\ & S'\ar[r]^\subset & S}$$
where \begin{itemize}
  \item $S=\Gamma_\Gbf\bsh X^+$ is a pure Shimura variety and $\pi:M\ra S$ is an abelian $S$-scheme defined by the natural projection $(\Pbf,Y)\ra(\Gbf,X)$;

  \item $S'$ is a (pure) special subvariety of $S$ defined by $(\Gbf',X')$ with $\Tbf$ equal to the connected center of $\Gbf'$,

  \item $M_{S'}$ is the abelian $S'$-scheme pullded-back from $M\ra S$, and $\Vbf'$ is a subrepresentation in $\Vbf$ of $\Gbf$, corresponding to an abelian subscheme $A'$ of $M_{S'}\ra S'$;

  \item $M'$ is a translation of $A'$ by a torsion section of $M_{S'}\ra S'$ given by $v$.
\end{itemize} In particular, this notion is more restrictive than $\Tbf$-special subvarieties studied in \cite{chen kuga}.

The case in general mixed Shimura varieties is similar. We show that certain spaces of probability measures on $M$ associated to $\TW$-special subvarieties are compact for the weak topology, from which we deduce the equidistribution of the supports of such measures, as well as the Andr\'e-Oort conjecture for such sequences of $\TW$-special subvarieties.

We  propose a notion of $B$-bounded sequences of special subvarieties, which means special subvarieties that are $\TW$-special when $\TW$ comes from some prescribed finite set $B$ of pairs $\TW$ as above. The main result of \cite{ullmo yafaev} shows in the pure case that a sequences with bounded Galois orbits is $B$-bounded for some $B$. 
The mixed case of this characterization is treated in \cite{chen bounded 2}


\bigskip

The paper is organized as follows:

In Section 1, we recall the basic notions of mixed Shimura data, their subdata, mixed Shimura varieties, their special subvarieties, as well as their connected components. We do not follow the original presentation in \cite{pink thesis} of the notion of mixed Shimura data; rather we use an equivalent formulation which is more convenient in our case, based on a fibration of a mixed Shimura variety over a pure Shimura variety, whose fibers are torus bundles over abelian varieties.

In Section 2, we  introduce some measure-theoretic objects, such as  lattice (sub)spaces,   S-(sub)spaces, and canonical probability measures associated to special subvarieties in mixed Shimura varieties. The notion of S-spaces is only involved for mixed Shimura varieties that are not Kuga varieties. We realize S-spaces as some real analytic subspaces of the corresponding special subvarieties, dense for the Zariski topology. They are fibred over pure Shimura varieties, and the fibers are compact, which allows us to define canonical probability measures supported by them.  We also introduce the notion of $B$-bounded sequence of special subvarieties, where $B$ is a finite set of pairs of the form $(\Tbf,w)$ explained above.

In Section 3, we prove the equidistribution of bounded sequences of special lattice subspaces and special S-spaces  . The proof is reduced to the case when the bound $B$ consists of a single element $\TW$, and the arguments are completely parallel to the pure $\Cbf$-special case in \cite{clozel ullmo} and \cite{ullmo yafaev}. The equidistribution of $B$-bounded S-subspaces implies the Andr\'e-Oort conjecture for a $B$-bounded sequence  of special subvarieties in a mixed Shimura variety.




\bigskip

\footnote: the author is partially supported by National Key Basic Research Program of China, No. 2013CB834202, the Grant of Chinese Universities Project WK0010000029, and the Youth Fund Project () of National Science Foundation of China.

\section*{notations and conventions}

Over a base field $k$, a linear $k$-group $\Hbf$ is a smooth affine algebraic $k$-group scheme, and $\Tbf_\Hbf$ is the connected center of $\Hbf$, namely the neutral component of the center of $\Hbf$.

We write $\Sbb$ for the Deligne torus $\Res_{\Tbf/\Rbb}\mult_\Cbb$. The ring of finite adeles is denoted by $\adele$. $\ibf$ is a fixed square root of -1 in $\Cbb$.

A linear $\Qbb$-group is compact if its set of $\Rbb$-points form a compact Lie group. For $\Pbf$ a linear $\Qbb$-group with maximal reductive quotient $\Pbf\epim\Gbf$, we write $\Pbf(\Rbb)^+$ resp. $\Pbf(\Rbb)_+$ for the preimage of $\Gbf(\Rbb)^+$ resp. of $\Gbf(\Rbb)_+$, in the sense of \cite{deligne pspm}. $\Pbf(\Rbb)^+$ is just the neutral component of the Lie group $\Pbf(\Rbb)$.

For a real or complex analytic space (not necessarily smooth), its archimedean topology is the one locally deduced from the archimedean metric on $\Rbb^n$ or $\Cbb^m$.

For $\Hbf$ a linear $\Qbb$-group, we write $\Xfrak(\Hbf)$ for the set of $\Rbb$-group homomorphisms $\Sbb\ra\Hbf_\Rbb$, and $\Yfrak(\Hbf)$ for the set of $\Cbb$-group homomorphisms $\Sbb_\Cbb\ra\Hbf_\Cbb$. We have the natural action of $\Hbf(\Rbb)$ on $\Xfrak(\Hbf)$ by conjugation, and similarly the action of $\Hbf(\Cbb)$ on $\Yfrak(\Hbf)$. In particular, we have an inclusion $\Xfrak(\Hbf)\mono\Yfrak(\Hbf)$, equivariant \wrt the inclusion $\Hbf(\Rbb)\mono\Hbf(\Cbb)$.

\section{preliminaries on (fibred) mixed Shimura varieties}

The readers are referred to \cite{pink thesis} for generalities on mixed Shimura data and mixed Shimura varieties. In our case, we mainly consider mixed Shimura data (or varieties) fibred over some pure Shimura data (or varieties). Therefore it is more convenient to use the following variant:

\begin{definition}[fibred mixed Shimura data, cf. \cite{chen kuga} 2.5]\label{fibred mixed Shimura data}

 A fibred mixed Shimura datum is a tuple $(\Gbf,X;\Ubf,\Vbf,\psi)$ as follows:

(a) a pure Shimura data $(\Gbf,X)$; in particular, $X$ is a $\Gbf(\Rbb)$-orbit in $\Xfrak(\Gbf)$ subject to some algebraic constraints.

(b) two finite dimensional algebraic representations of linear $\Qbb$-groups  $\rho_U:\Gbf\ra\GL_\Ubf$ and $\rho_\Vbf:\Gbf\ra\GL_\Vbf$, $\Ubf$ and $\Vbf$ being viewed as vectorial $\Qbb$-groups, such that for any $x\in X$, the composition $\rho_\Ubf\circ x$ defines a rational pure Hodge structure of type $\{(-1,-1)\}$, and that $\rho_\Vbf\circ x$ defines a rational pure Hodge structure of type $\{(-1,0),(0,-1)\}$.

(c) a $\Gbf$-equivariant alternative bilinear map $\psi:\Vbf\times\Vbf\ra\Ubf$; it is equivalently given by a central extension of unipotent linear $\Qbb$-group $$1\ra\Ubf\ra\Wbf\ra\Vbf\ra 1,$$ respecting the $\Gbf$-actions as the Lie bracket $[\ ,\ ]:\Wbf\times\Wbf\ra\Wbf$ factors through $2\psi:\Vbf\times\Vbf\ra\Ubf$.

(d) The action of the connected center $\Tbf_\Gbf$ on $\Vbf$ and on $\Ubf$ factors through $\Qbb$-tori isogeneous to products of the form $\Hbf\times\mult^d$, where $\Hbf$ is a compact $\Qbb$-torus.

(e) We also require $\Gbf$ to be minimal, in the sense that if $\Hbf\subset\Gbf$ is a $\Qbb$-subgroup such that $x(\Sbb)\subset\Hbf_\Rbb$ for all $x\in X$ then $\Hbf=\Gbf$.

\bigskip

The classical mixed Shimura datum in the sense of \cite{pink thesis} associated to the tuple above is $(\Pbf,Y)$ where $\Pbf=\Wbf\rtimes\Gbf$ is given by the action of $\Gbf$ on $\Wbf$ (via $\rho_\Ubf$ and $\rho_\Vbf$), and $Y$ is the $\Ubf(\Cbb)\Wbf(\Rbb)$-orbit of $X$ in $\Yfrak(\Pbf)$. In particular, $\Wbf$ is the unipotent radical of $\Pbf$,  $\Ubf$ is the center of $\Wbf$, and $Y$ is fibred over $X$ by $\Ubf(\Cbb)\Wbf(\Rbb)$..

We thus write $(\Pbf,Y)=(\Ubf,\Vbf)\rtimes(\Gbf,X)$ for the fibred mixed Shimura datum, with $(\Gbf,X)$ referred to as its pure base, and we often write $\Wbf=(\Ubf,\Vbf)$ as the unipotent radical.
\end{definition}

\begin{remark}\label{remarks on mixed shimura data}

(1) In the definition above, one can also follow  \cite{pink thesis} for a slightly different notion of pure Shimura datum $(\Gbf,X)$. The connected mixed Shimura varieties thus obtained is the same as the ones defined by \ref{fibred mixed Shimura data}, cf. \cite{chen kuga} Remark 2.2(1).

(2) The condition (d)  in included as it was required in \cite{pink thesis} to establish the existence of canonical models, which we need when considering Galois orbits of special subvarieties.

(3) The condition (e) is included to simplify some formulations. Removing this condition does not cause any difference on the space $X$ nor the mixed Shimura varieties thereby defined. It is used in \cite{klingler yafaev} and \cite{ullmo yafaev}, whose approach we follow closely.

(4) One can also show that $\Gbf^\der$ acts on $\Ubf$ trivially via $\rho_\Ubf$, cf. \cite{pink thesis} 2.16.

(5) We briefly outline how a fibred mixed Shimura datum can be constructed out of a classical mixed Shimura dautm. By definition (\cite{pink thesis} 2.1), if $(\Pbf,Y)$ is a classical mixed Shimura datum, then for any $y\in Y$, the homomorphism $y:\Sbb_\Cbb\ra(\Pbf/\Ubf)_\Cbb$ deduced from $y$ is already defined over $\Rbb$. Since the image of $y$ is a $\Cbb$-torus, take a $\Ubf(\Cbb)\Wbf(\Rbb)$-conjugation, we may assume that $y$ has image in  $\Gbf_\Cbb$, for some Levi $\Qbb$-subgroup $\Gbf$ of $\Pbf$. Then $y$ is defined over $\Rbb$, and factors as $y:\Sbb\ra\Gbf_\Rbb\ra\Pbf_\Rbb$. Put $X=\Gbf(\Rbb)\cdot y\subset Y$ the orbit under $\Gbf(\Rbb)$, one can verify that $(\Gbf,X)$ is a pure Shimura datum, and that the Levi decomposition $\Pbf=\Wbf\rtimes\Gbf$ induces an equality $Y=\Ubf(\Cbb)\Wbf(\Rbb)\cdot X$ so that $(\Pbf,Y)$ is the classical mxied Shimura datum associated to $(\Ubf,\Vbf)\rtimes(\Gbf,X)$.

\end{remark}

\begin{definition}[morphisms of mixed Shimura data]\label{morphism of mixed shimura data}

(1) A morphism between fibred mixed Shimura is of the form $(f,f_*):(\Ubf,\Vbf)\rtimes(\Gbf,X)\ra(\Ubf',\Vbf')\rtimes(\Gbf',X')$, where $f$ means a homomorphism $f:\Gbf\ra\Gbf'$ together with equivariant homomorphism $\Vbf'\ra\Vbf', \Ubf\ra\Ubf'$ such that $f_*:\Xfrak(\Gbf)\ra\Xfrak(\Gbf'),\ x\mapsto f\circ x$ sends $X$ into $X'$.

From a morphism between fibred data $(f,f_*)$ as above we immediately get a morphism between classical mixed Shimura data: the $\Qbb$-group homomorphism $\Pbf\ra\Pbf'$  is the evident one, denoted still by $f$; the induced map $f_*:\Xfrak(\Gbf)\ra\Xfrak(\Gbf')$ extends to $\Yfrak(\Pbf)\ra\Yfrak(\Pbf')$ using unipotent extensions by $\Wbf$ and $\Wbf'$ respectively, and it sends the $\Ubf(\Cbb)\Wbf'(\Rbb)$-orbit of $X$ into the $\Ubf'(\Cbb)\Wbf(\Rbb)$-orbit of $X'$, hence the required equivariant map $Y\ra Y'$

\bigskip

(2) A subdatum of $(\Pbf,Y)$ is given by a morphism of mixed Shimura data $(\Ubf',\Vbf')\rtimes(\Gbf',X')\ra(\Ubf,\Vbf)\rtimes(\Gbf,X)$ such that the corresponding maps $\Pbf':=\Wbf'\rtimes\Gbf'\ra\Pbf=\Wbf\rtimes\Gbf$ and $Y':=\Ubf'(\Cbb)\Wbf'(\Rbb)X'\ra Y=\Ubf(\Cbb)\Wbf(\Rbb)X$ are both inclusions.

One can verify that if $(f,f_*): (\Pbf',Y')\ra(\Pbf,Y)$ is a morphism of mixed Shimura data, then $(f(\Pbf'),f_*(Y'))$ is a subdatum of $(\Pbf,Y)$, which we call the image of $(f,f_*)$.

\bigskip

(3) For $(\Pbf,Y)=(\Ubf,\Vbf)\rtimes(\Gbf,X)$ a mixed Shimura datum, and $\Nbf$ a normal unipotent $\Qbb$-subgroup of $\Pbf$ (necessarily contained in $\Wbf$), the (unipotent) quotient $(\Pbf,Y)/\Nbf$ is the datum $(\Pbf',Y')=(\Ubf',\Vbf')\rtimes(\Gbf,X)$, where $\Ubf'=\Ubf/(\Nbf\cap\Ubf)$ and $\Vbf'=(\Wbf/\Nbf)/\Ubf'$, and $Y'$ is the quotient manifold of $Y$ by the free action of $(\Nbf\cap\Ubf)(\Cbb)\Nbf(\Rbb)$.

The unipotent quotient here can be formulated differently, see \cite{pink thesis} 2.9 and 2.18.
\bigskip

(4) A fibred Kuga datum is defined as a a fibred mixed Shimura datum $(\Ubf,\Vbf)\rtimes(\Gbf,X)$ such that $\Ubf=0$.  Similarly, a pure Shimura datum is a (fibred) mixed Shimura datum such that $\Ubf=0=\Vbf$.

It is then clear that unipotent quotient by $\Ubf$ and by $\Vbf$ successively gives a two-step fibration $$(\Pbf, Y)\ra(\Pbf/\Ubf, Y/\Ubf(\Cbb))\ra(\Pbf/\Wbf, Y/\Ubf(\Cbb)\Wbf(\Rbb))$$ where the middle term is Kuga, and the last term is the maximal pure quotient.

It is also clear that the subdata of a pure Shimura datum are pure, the subdata of a Kuga datum are Kuga (including the pure ones), etc. as one can verify directly using the Hodge types.

\bigskip

(5) In particular we have the inclusion $(\Gbf,X)\subset(\Pbf,Y)$, which we call the pure section given by the Levi decomposition $\Pbf=\Wbf\rtimes\Gbf$. For any $w\in\Wbf(\Qbb)$, we have a pure subdatum $(w\Gbf w^\inv, wX)$ in $(\Pbf,Y)$ isomorphic to $(\Gbf,X)$; all the maximal pure subdata of $(\Pbf,Y)$ are of this form, due to the Levi decomposition $\Pbf=\Wbf\rtimes\Gbf$.
\end{definition}

We often omit the adjective "fibred" if no confusion is caused.

\begin{notation}\label{group law} We also fix notations for the group law in $\Pbf=\Wbf\rtimes\Gbf$. $\Wbf$ is isomorphic to $\Ubf\times\Vbf$ as $\Qbb$-scheme. Then the group law writes as \begin{itemize}
  \item multiplication $(u,v)\times(u',v')=(u+u'+\psi(v,v'),v+v')$;

  \item inverse $(u,v)^\inv=(-u,-v)$;

  \item commutator $(u,v)(u',v')(-u,-v)(-u',-v')=(2\psi(v,v'),0)$.

 \end{itemize}
 Elements in $\Pbf=\Wbf\rtimes\Gbf$ are written as $(u,v,g)$, with neutral element $(0,0,1)$, and we have \begin{itemize}
   \item multiplication $(u,v,g)(u',v',g')=(u+g(u')+\psi(v,g(v')),v+g(v'),gg')$;
   \item inverse $(u,v,g)^\inv=(-g^\inv(u) , -g^\inv(v), g^\inv)$, namely $(w,g)^\inv=(g^\inv(w^\inv),g^\inv)$ for $w=(u,v)$
   \item and the commutator between $\Wbf$ and $\Gbf$ is $(u,v,1)(0,0,g)(-u,-v,1)(0,0,g^\inv)=(u-g(u),v-g(v),1)$

 \end{itemize}
 where we write $g(u)=gug^\inv=\rho_\Ubf(g)(u)$ and similarly for $g(v)$, $g(w)$.
\end{notation}

\begin{definition-proposition}[description of subdata, cf. \cite{chen kuga} Proposition 2.6, 2.10]\label{description of subdata}

Let $(\Pbf,Y)=(\Ubf,\Vbf)\rtimes(\Gbf,X)$ be a mixed Shimura datum. Then a subdatum of $(\Pbf,Y)$ is of the form $(\Pbf',Y')=(\Ubf',\Vbf')\rtimes(w\Gbf'w^\inv, wX)$ where \begin{itemize}

\item $(\Gbf',X')\subset(\Gbf,X)$ is a pure Shimura subdatum;

\item $\Ubf'$ resp. $\Vbf'$ is a subrepresentation of $\rho_U$ resp. of $\rho_\Vbf$ restricted to $\Gbf'$ such that $\psi(\Vbf'\times\Vbf')\subset\Ubf'$;

\item $w\in\Wbf(\Qbb)$ conjugates $(\Gbf',X')$ into a pure subdatum of $(\Pbf,Y)$, where $w X'$ is the translate of $X'$ in $Y$ by $w$; $w\Gbf'w^\inv$ acts on $\Ubf'$ and $\Vbf'$ through $\Gbf'$.

\end{itemize}
\end{definition-proposition}

It is also clear that maximal pure subdata of $(\Pbf,Y)$ are of the form $(w\Gbf w^\inv, wX)$, and we call them pure sections of $(\Pbf,Y)$ (\wrt the natural projection).

\begin{proof}[sketch of the proof]
The idea  is  the same as the Kuga case in \cite{chen kuga} Proposition 2.10.

 \begin{itemize}

  \item Following \ref{remarks on mixed shimura data}(5), we see that $(\Pbf',Y')$ is isomorphic to $(\Ubf',\Vbf')\rtimes(\Gbf_1,X_1)$ for some pure Shimura datum $(\Gbf_1,X_1)$, and $(\Gbf_1,X_1)$ extends to a maximal pure subdatum of $(\Pbf,Y)$, which is given by a Levi $\Qbb$-subgroup of $\Pbf$. Thus $(\Gbf_1,X_1)$ is a pure subdatum of  $(w_1\Gbf w_1^\inv, w_1X)$ for some $w_1\in\Wbf(\Qbb)$.

  \item  The restrictions on Hodge types shows that $\Ubf'\subset\Ubf$ and $\Vbf'\subset\Vbf$, and they are subrepresentations of $\Gbf_1$ in $\Ubf$ and $\Vbf$ respectively.

  \item Finally, using the notations in \ref{group law} one verifies directly that the action of $w\Gbf w^\inv$ on $\Wbf$ is the same as the action of $\Gbf$, i.e. $$(w g(w^\inv),g)(w',1)(wg(w^\inv),g)^\inv=(g(w),1).$$
\end{itemize}
\end{proof}

\begin{remark}[pure sections]\label{remak on pure sections}
In terms of classical mixed Shimura data, we have a commutative diagram of morphisms $$\xymatrix{ (\Pbf',Y')\ar[d] \ar[r]^\subset & (\Pbf,Y)\ar[d]\\ (\Gbf',X')\ar[r]^\subset & (\Gbf,X)}$$ where the horizontal arrows are inclusions of subdata, and the vertical ones are natural projections. The pure section $(\Gbf,X)\mono(\Pbf,Y)$ does not restrict to  a pure section of $(\Pbf',Y')$ but we can extend a pure section of $(\Pbf',Y')$ to a maximal pure subdatum of $(\Pbf,Y)$, which is a pure section conjugate to $(\Gbf,X)$.
\end{remark}

\begin{definition}[mixed Shimura varieties, cf. \cite{pink thesis} 3.1]\label{mixed shimura varieties}

Let $(\Pbf,Y)=(\Ubf,\Vbf)\rtimes(\Gbf,X)$ be a mixed Shimura datum, and let $K\subset\Pbf(\adele)$ a \cosg. The (complex) mixed Shimura variety associated to $(\Pbf,Y)$ at level $K$ is the quotient space $$M_K(\Pbf,Y)(\Cbb)=\Pbf(\Qbb)\bsh[ Y\times\Pbf(\adele)/K]\isom\Pbf(\Qbb)_+\bsh[ Y^+\times\Pbf(\adele)/K]$$ where the last equality makes sense for any connected component $Y^+$ of $Y$ because $\Pbf(\Qbb)_+$ equals the stabilizer of $Y^+$ in $\Pbf(\Qbb)$.

Using the finiteness of class number  in \cite{conrad class number}, we see that the double quotient $\Pbf(\Qbb)_+\bsh\Pbf(\adele)/K$ is finite. Write $\Rcal$ for a set of representative of it, we then have $$M_K(\Pbf,Y)=\coprod_{a\in\Rcal}\Gamma_K(a)\bsh Y^+$$ with $\Gamma_K(a)=\Pbf(\Qbb)_+\cap aKa^\inv$ a congruence subgroup of $\Pbf(\Rbb)_+$.

Pink has shown that such double quotient are normal quasi-projective varieties over $\Cbb$, generalizing  a theorem of Baily and Borel. He also develops the theory of canonical models to the pure case. In this paper, we treat mixed Shimura varieties as algebraic varieties over $\Qac$, and we denote them as $M_K(\Pbf,Y)$. When we need to mention a canonical model over some field of definition $F$, we will write $M_K(\Pbf,Y)_F$, which is NOT the base change to some ''smaller'' base field. We apologize for the abuse of notations.

Kuga varieties resp. pure Shimura varieties are mixed Shimura varieties associated to Kuga data resp. pure Shimura data.

\end{definition}

Since the Andr\'e-Oort conjecture is insensitive to changing the level $K$, we will mainly work with levels $K$ that are neat, see \cite{pink thesis} Introduction (page 5). Mixed Shimura varieties at neat levels are smooth.

\begin{theorem}[canonical model, cf. \cite{pink thesis} Chapter 11] The double quotient $M_K(\Pbf,Y)$ admits a canonical model over the reflex field of $(\Pbf,Y)$, unique up to a unique isomorphism.

\end{theorem}

In this paper, it suffices to know that the reflex field is a number field embedded in $\Cbb$, and that the following morphisms are functorially defined with respect to the canonical models over the reflex fields:

\begin{definition}[morphisms of mixed Shimura varieties and Hecke translates, cf. \cite{pink thesis} 4]\label{morphisms of mixed Shimura varieties}

 (1) Let $f:(\Pbf,Y)\ra(\Pbf',Y')$ be a morphism of mixed Shimura data, with \cosgs $K\subset\Pbf(\adele)$ and $K'\subset\Pbf'(\adele)$ such that $f(K)\subset K'$, then there exists a unique morphism $M_K(\Pbf,Y)\ra M_{K'}(\Pbf',Y')$ of mixed Shimura varieties  whose evaluation over $\Cbb$-points is simply $[x,aK]\mapsto[f_*(x), f(a)K']$. It is actually defined over the composite of the reflex fields of these data.
\bigskip

(2) Let $(\Pbf,Y)$ be a mixed Shimura datum, and $g\in\Pbf(\adele)$. For $K\subset\Pbf(\adele)$ a \cosg, there exists a unique isomorphism of mixed Shimura varieties $M_{gKg^\inv}(\Pbf,Y)\ra M_K(\Pbf,Y)$ which is $$[x,agKg^\inv]\mapsto[x,agK]$$ at the level of $\Cbb$-points. It is actually defined over the reflex field of the datum.

\bigskip

For $(\Pbf,X)=(\Ubf,\Vbf)\rtimes(\Gbf,X)$, the natural projection $\pi:(\Pbf,Y)\ra(\Gbf,X)$ gives the natural projection onto the pure Shimura variety $$\pi:M_K(\Pbf,Y)\ra M_{\pi(K)}(\Gbf,X).$$ We can refine this projection into $$M_K(\Pbf,Y)\ot{\pi_\Ubf}\lra M_{\pi_\Ubf(K)}(\Pbf/\Ubf,Y/\Ubf(\Cbb))\ot{\pi_\Vbf} \lra M_{\pi(K)}(\Gbf,X)$$ as $\pi=\pi_\Wbf=\pi_\Vbf\circ\pi_\Ubf$, and the sequence means that a general mixed Shimura variety is fibred over some Kuga variety.
\end{definition}

We also introduce an auxiliary condition of the \cosg $K$:

\begin{definition}[levels of product type]\label{levels of product type}

Let $(\Pbf,Y)=(\Ubf,\Vbf)\rtimes(\Gbf,X)$ be a fibred mixed Shimura datum. A \cosg $K$ of $\Pbf(\adele)$ is said to be of product type, if it is of the form $K=K_\Wbf\rtimes K_\Gbf$ for \cosgs $K_\Wbf\subset\Wbf(\adele)$, $K_\Gbf\subset\Gbf(\adele)$, with $K_\Wbf$ the central extension of a \cosg $K_\Vbf\subset\Vbf(\adele)$ by a cosg $K_\Ubf\subset\Ubf(\adele)$ through the restriction of $\psi$; $K_\Ubf$ and $K_\Vbf$ are required to be stabilized by $K_\Gbf$.

Furthermore, a \cosg $K$ in $\Pbf(\adele)$ is said to be of strong product type if \begin{itemize}
 
\item (a)  it is of product type and $K=\prod_p K_p$ for \cosgs $K_p\subset\Pbf(\Qbbp)$ for all rational prime $p$, such that for some $\wp$ prime, $K_\wp$ is neat;
    
\item (b) we also require that $K_\Gbf=K_{\Gbf^\der}K_\Cbf$ where $\Cbf$ is the connected center of $\Gbf$, with \cosgs $K_{\Gbf^\der}\subset\Gbf^\der(\adele)$ and $K_\Cbf\subset\Cbf(\adele)$ both of strong product type in the sense of (a).
\end{itemize}
In this case we also write $K_p=K_{\Wbf,p}\rtimes K_{\Gbf,p}$
and $K_?=\prod_p K_{?,p}$ for $?\in\{\Ubf,\Vbf,\Wbf,\Gbf,\Pbf\}$.
\end{definition}


\begin{remark}[two-step fibration, cf. \cite{pink thesis} ?.?]\label{two step fibration}
In this case, $\pi(K)=K_\Gbf$, and we have an evident morphism $\iota(0): M_{K_\Gbf}(\Gbf,X)\mono M_K(\Pbf,Y)$, which we called the zero section of the (fibred) mixed Shimura variety defined by $(\Pbf,Y)=(\Ubf,\Vbf)\rtimes(\Gbf,X)$. It can be refined into $$ M_K(\Pbf,Y)\ot{\pi_\Ubf}\lra M_{K_\Vbf\rtimes K_\Gbf}(\Vbf\rtimes(\Gbf,X))\ot{\pi_\Vbf}\lra M_{K_\Gbf}(\Gbf,X)$$ where $\pi_\Vbf$ is an abelian scheme with zero section $\pi_\Ubf\circ \iota(0)$, and $\pi_\Ubf$ is a torsor under $\Gamma_\Ubf\bsh\Ubf(\Cbb)$ which is the algebraic torus of character group $\Gamma_\Ubf$.  The extension $\Wbf$ of $\Vbf$ by $\Ubf$ splits \ifof the torsor admits a section given by a morphism between mixed Shimura varieties.

\end{remark}

\begin{example}

Let $\Vbf$ be a finite-dimensional $\Qbb$-vector space, equipped with a symplectic form $\psi:\Vbf\times\Vbf\ra\Qbb(-1)$ where $\Qbb(-1)$ is the $\Qbb$-vector space $\Qbb\dfrac{1}{2\pi\ibf}$. From the symplectic $\Qbb$-group $\GSp_\Vbf=\GSp(\Vbf,\psi)$, we obtain the pure Shimura datum $(\GSp_\Vbf,\Hscr_\Vbf)$, with $\Hscr_\Vbf$ the Siegel double space associated to $(\Vbf,\psi)$. We refer to it as the Siegel datum associated to $(\Vbf,\psi)$. When $(\Vbf,\psi)$ is the standard symplectic structure on $\Qbb^{2g}$, we simply write it as $(\GSp_{2g},\Hscr_g)$.

For any $x\in\Hscr_\Vbf$, the standard representation $\rho_\Vbf:\GSp_\Vbf\ra\GL_\Vbf$ defines a rational Hodge structure $(\Vbf,\rho_\Vbf\circ x)$ of type $\{(-1,0),(0,-1)\}$, hence we get the Kuga-Siegel datum $\Vbf\rtimes(\GSp_\Vbf,\Hscr_\Vbf)$.

The symplectic form defines a central extension $\Wbf$ of $\Vbf$ by $\Qbb(-1)$, and it is easy to verify that $(\Pbf_\Vbf,Y_\Vbf):=(\Qbb(-1),\Vbf)\rtimes(\GSp_\Vbf,\Hscr_\Vbf)$ is a mixed Shimura datum fibred over the Siegel datum.

Assume that for some $\Zbb$-lattice $\Gamma_\Vbf$ of $\Vbf$, the restriction $\psi:\Gamma_\Vbf\times\Gamma_\Vbf\ra\Qbb(-1)$ has value in $\Zbb(-1)$ and is of discriminant $\pm1$. The lattices $\Gamma_\Vbf\subset\Vbf$ and $\Zbb(-1)\subset\Qbb(-1)$ generates \cosgs $K_\Vbf$ and $K_\Ubf$ respectively. Take a \cosg $K_\Gbf\subset\GSp_\Vbf(\adele)$ small enough and stabilizing both  $K_\Vbf$ and $K_\Ubf$, we get the mixed Shimura variety $M_K(\Pbf_\Vbf,Y)$ for $(\Pbf,Y)=(\Pbf_\Vbf,Y_\Vbf)$ and $K=K_\Wbf\rtimes K_\Gbf$, $K_\Wbf$ is the \cosg generated by $K_\Ubf$ and $K_\Vbf$. We also have the universal abelian scheme over the Siegel moduli space of level $K_\Gbf$, namely $$\Ascr=M_{K_\Vbf\rtimes K_\Gbf}(\Vbf\rtimes\GSp_\Vbf,\Vbf(\Rbb)\Hscr_\Vbf)\ra \Sscr=M_{K_\Gbf}(\GSp_\Vbf,\Hscr_\Vbf)$$ and $M_K(\Pbf,Y)$ is a $\mult$-torsor over $\Ascr$.

The \cosgs thus obtained are also levels of product type.

\end{example}

\begin{definition}[special subvarieties]\label{special subvarieties}

Let $(\Pbf,Y)$ be a mixed Shimura datum, with $M=M_K(\Pbf,Y)$ a mixed Shimura variety associated to it.

(1) The map $\wp_\Pbf:Y\times\Pbf(\adele)/K\ra M(\Cbb),\ (y,aK)\mapsto [y,aK]$ is called the (complex) uniformization map of $M$. It is clear that the target is not connected, but its connected components are simply connected complex manifolds isomorphic to each other.

A special subvariety of $M_K(\Pbf,Y)$ is a priori a subset of $M(\Cbb)$ of the form $\wp_\Pbf( Y'^+\times aK)$ with $a\in\Pbf(\adele)$ and $Y'^+$ is a connected component of some mixed Shimura subdatum $(\Pbf',Y')\subset(\Pbf,Y)$.

A special subvariety is actually a closed algebraic subvariety of $M_K(\Pbf,Y)$: it is a connected component of the image of the morphism $M_{K'}(\Pbf',Y')\ra M_{aKa^\inv}(\Pbf,Y)$ under the Hecke translate $M_{aKa^\inv}(\Pbf,Y)\isom M_K(\Pbf,Y)$; here $K'=\Pbf'(\adele)\cap aKa^\inv$.

(2) In Section 2 and 3, we will often with connected mixed Shimura varieties defined as follows:

\begin{itemize}

\item a connected mixed Shimura datum is of the form $(\Pbf,Y;Y^+)$ where $(\Pbf,Y)$ is a mixed Shimura datum and $Y^+$ a connected component of $Y$; a connected mixed Shimura subdatum is of the form $(\Pbf',Y';Y'^+)\subset(\Pbf,Y;Y^+)$ with $Y'^+\subset Y^+$

\item a connected mixed Shimura variety is a quotient space of the form $M^+=\Gamma\bsh Y^+$ where $\Gamma\subset\Pbf(\Qbb)_+$ is a congruence subgroup; such quotients are normal quasi-projective algebraic varieties defined over a finite extension of the reflex field of $(\Pbf,Y)$, and we treat them as varieties over $\Qac$;

\item for a connected mixed Shimura variety $M^+$ as above we have the (complex) uniformization map $\wp_\Gamma: Y^+\ra M^+$, and a special subvariety of $M^+$ is a subset of the form $\wp_\Gamma(Y'^+)$ given by some connected mixed Shimura subdatum $(\Pbf',Y';Y'^+)$; special subvarieties are closed irreducible algebraic subvarieties defined over $\Qac$, with a canonical model defined over some number field.

\end{itemize}

For example, in the Kuga case $(\Pbf,Y)=\Vbf\rtimes(\Gbf,X)$,  we have explained in the Introduction that special subvarieties are certain torsion subschemes of abelian schemes pulled-back from $S=\Gamma_\Gbf\bsh X^+$ to some pure special subvariety $S'\subset S$ (and the conditons involving $\TW$ describes some finer properties of the special subvariety).


\end{definition}









\section{measure-theoretic constructions associated mixed Shimura varieties}

In this section, we introduce some measure-theoretic constructions  associated to connected mixed Shimura varieties. Most of them are analogue to the Kuga case discussed in \cite{chen kuga} Section 2, 2.14-2.18, except that in the general case, we work with the notion of S-spaces. We also introduce the notion of $\TW$-special subdata, which are analogue of $\Tbf$-special subdata in the mixed case.

\begin{definition}[lattice spaces and canonical measures]\label{lattice spaces and canonical measures}

(1) A linear $\Qbb$-group $\Pbf$ is said to be of type $\Hscr$ if it is of the form $\Pbf=\Wbf\rtimes\Hbf$ with $\Wbf$ a unipotent $\Qbb$-group and $\Hbf$ a connected semi-simple $\Qbb$-group without normal $\Qbb$-subgroups $\Hbf'\subset\Hbf$ of dimension $>0$ such that $\Hbf'(\Rbb)$ is compact.

We show in the lemma \ref{commutator} below that for a mixed Shimura datum $(\Pbf,Y)$, the $\Qbb$-group of commutators  $\Pbf^\der$ is of type $\Hscr$.

(2) For $\Pbf$ a linear group of type $\Hscr$ and $\Gamma\subset\Pbf(\Rbb)^+$ a congruence subgroup, the quotient $\Omega=\Gamma\bsh \Pbf(\Rbb)^+$ is called the (connected) lattice space associated to $(\Pbf,\Gamma)$. Since $\Gamma$ is discrete in $\Pbf(\Rbb)^+$, the space $\Omega$ is a smooth manifold. We also have the uniformization map $\wp_\Gamma:\Pbf(\Rbb)^+\ra\Omega,\ a\mapsto\Gamma a$.

(3) Let $\Omega=\Gamma\bsh\Pbf(\Rbb)^+$ be a lattice space as in (2). The left Haar measure $\nu_\Pbf$ on $\Pbf(\Rbb)^+$ passes to a measure $\nu_\Omega$ on $\Omega$: choose a fundamental domain $F\subset\Pbf(\Rbb)^+$ \wrt $\Gamma$, we put $\nu_\Omega(A)=\nu_\Pbf(F\cap\wp_\Gamma^\inv A)$ for $A\subset\Omega$ measurable.

Following \cite{chen kuga} 2.15 (1), $\nu_\Omega$ is of finite volume and is normalized such that $\nu_\Omega(\Omega)=1$.

\end{definition}

\begin{lemma}\label{commutator} Let $(\Pbf,Y)=(\Ubf,\Vbf)\rtimes(\Gbf,X)$ be a mixed Shimura datum. Then

(1) $\Pbf^\der=\Wbf\rtimes\Gbf^\der$ is of type $\Hscr$, where $\Wbf$ is the central extension of $\Vbf$ by $\Ubf$;

(2) $\Gbf^\der$ acts on $\Ubf$ trivially.

\end{lemma}

\begin{proof}

(1) One may argue as in \cite{chen kuga} as the representations $\rho_\Vbf:\Gbf\ra\GL_\Vbf$ and $\rho_\Ubf:\Gbf\ra\GL_\Ubf$ admits no trivial subrepresentations due to the constraints on Hodge types.

(2) This is proved in \cite{pink thesis} 2.16.
\end{proof}

\begin{definition}[lattice space and S-space]\label{lattice space and s-space}

Let $(\Pbf,Y;Y^+)=(\Ubf,\Vbf)\rtimes(\Gbf,X;X^+)$ be a connected mixed Shimura datum with pure section $(\Gbf,X;X^+)$. Let $\Gamma$ be a congruence subgroup of $\Pbf(\Rbb)_+$, which gives us the connected mixed Shimura variety $M=\Gamma\bsh Y^+$.

(1) The lattice space associated to $M$ is $\Omega=\Gamma^\dagger\bsh \Pbf^\der(\Rbb)+$ where $\Gamma^\dagger=\Gamma\cap\Pbf^\der(\Rbb)^+$. It is equipped with the canonical measure $\nu_\Omega$, and we have the uniformazation map $\wp_\Gamma:\Pbf(\Rbb)^+\ra \Omega$.

A lattice subspace of $\Omega$ is of the form $\wp_\Gamma(\Hbf(\Rbb)^+$ for some $\Qbb$-subgroup $\Hbf\subset\Pbf^\der$ of type $\Hscr$.
\bigskip

(2) We write $\Yscr^+$ for the $\Pbf(\Rbb)_+$-orbit of $X^+$ in $Y$, called the real part of $Y^+$. The (connected) S-space associated to $M$ is $\Mscr=\Gamma\bsh \Yscr^+$. Since $\Gamma$ contains a neat subgroup of finite index, the quotient $\Mscr$ is a real analytic space with at most singularities of finite group quotient.

We also have the following orbit map associated to any point $y\in\Yscr^+$:

$$\kappa_y:\Omega=\Gamma^\dagger\bsh\Pbf^\der(\Rbb)^+\ra \Mscr=\Gamma\bsh \Yscr^+,\ \Gamma^\dagger g\mapsto \Gamma gy.$$ It is surjective because $\Gamma^\dagger\subset\Gamma$ and $\Yscr^+$ is a single $\Pbf^\der(\Rbb)^+$-orbit, as $X^+$ is homogeneous under $\Gbf^\der(\Rbb)^+$. It is a submersion of smooth real analytic spaces when $\Gamma$ is neat.

\end{definition}

\begin{remark}\label{remarks on s-spaces}

(1) The $\Pbf(\Rbb)$-orbit $\Yscr$ of $X$ in $Y$ is independent of the choice of pure section $(\Gbf,X)$ as different pure sections are conjugate to each other under $\Pbf(\Qbb)\subset\Pbf(\Rbb)$. $\Yscr^+$ is simply a connected component of $\Yscr$, as it is the pre-image of $X^+$ under the natural projection $\Yscr(\mono Y)\epim X$ whose fibers are connected (isomorphic to $\Wbf(\Rbb))$.

(2) In the Kuga case, we have $\Ubf=0$, hence the real part $\Yscr^+$ equals $Y^+$, and the S-space $\Mscr$ is just the Kuga variety. In the non-Kuga case, the projection $\pi_\Ubf$ gives us the commutative diagram $$\xymatrix{\Yscr^+\ar[d]^{\pi_\Ubf}\ar[r]^\subset & Y^+\ar[d]^{\pi_\Ubf}\\ Y^+/\Ubf(\Cbb)\ar[r]^= & Y^+/\Ubf(\Cbb)},$$ in which the vertical maps are smooth submersions of manifolds. The fibers of the right vertical map are $\Ubf(\Cbb)$, while fibers on the left are $\Ubf(\Rbb)$.

\end{remark}

\begin{lemma}

Let $(\Pbf,Y;Y^+)=(\Ubf,\Vbf)\rtimes(\Gbf,X;X^+)$ be a mixed Shimura datum with $\Ubf\neq 0$. Then for any congruence subgroup $\Gamma\subset\Pbf(\Qbb)_+$, the S-space $\Gamma\bsh \Yscr^+$ is dense in $\Gamma\bsh Y^+$ for the Zariski topology.

\end{lemma}

\begin{proof}
In the projection $\pi_\Ubf: (\Ubf,\Vbf)\rtimes(\Gbf,X)\ra \Vbf\rtimes(\Gbf,X)$, $\pi_\Ubf:Y\ra Y/\Ubf(\Cbb)$ is a submersion of complex manifolds, and it is a $\Ubf(\Cbb)$-torsor. The subspace $\Yscr$ is the orbit of $X$ under $\Pbf(\Rbb)$ while $Y/\Ubf(\Cbb)$ is the orbit of $X$ under $\Vbf(\Rbb)\rtimes\Gbf(\Rbb)$. So $\Yscr$ is a $\Ubf(\Rbb)$-torsor over $Y/\Ubf(\Rbb)$, lying inside $Y$. $\Ubf(\Rbb)\subset\Ubf(\Cbb)$ is Zariski dense when we view $\Ubf(\Cbb)$ as a complex algebraic variety, hence the density of $\Yscr$ in $Y$. The proof for $\Mscr\subset M$ is clear when we restrict to connected components and take quotient by $\Gamma$.

\end{proof}

\begin{remark}[compact tori vs. algebraic tori]\label{compact tori vs. algebraic tori}
When $\Vbf=0$ and $\Gamma=\Gamma_\Ubf\rtimes\Gamma_\Gbf$ for some lattice $\Gamma_\Ubf$ in $\Ubf(\Qbb)$ stabilized by $\Gamma_\Gbf$ a congruence subgroup of $\Gbf(\Qbb)_+$, the fibration $M=\Gamma\bsh Y^+\ra S=\Gamma_\Gbf\bsh X^+$ is a torus group scheme over $S$, whose fibers are complex tori isomorphic to $\Gamma_\Ubf\bsh\Ubf(\Cbb)\isom\Zbb^d\bsh\Cbb^d$, $d$ being the dimension of $\Ubf$. Thus the S-space $\Mscr=\Gamma\bsh \Yscr^+$ is a real analytic subgroup of $M$ relative to the base $S$, whose fibers are compact tori isomorphic to $\Gamma_\Ubf\bsh \Ubf(\Rbb)$ in the split comlex tori $\Gamma_\Ubf\bsh \Ubf(\Cbb)$, hence Zariski dense.

Using harmonic analysis on $\Gamma_\Ubf\bsh \Ubf(\Rbb)$ one can prove that the analytic closure of a sequence of connected closed Lie subtori in it is still a connected closed Lie subtorus, which implies that the Zariski closure of a sequence of connected algebraic subtori in $\Gamma_\Ubf\bsh\Ubf(\Cbb)$ is an algebraic subtorus, cf. \cite{ratazzi ullmo} Section 4.1. This can be viewed as a motivation for our notion of S-spaces.

\end{remark}

The advantage of S-spaces is that they carry canonical measures of finite volumes. Parallel to the Kuga case studied in \cite{chen kuga} 2.17 and 2.18, we have the following:

\begin{definition-proposition}[canonical measures on S-spaces]\label{canonical measures on s-spaces}

Let $M=\Gamma\bsh Y^+$ be a connected mixed Shimura variety associated to $(\Pbf,Y;Y^+)$, with $\Omega=\Gamma^\dagger\bsh\Pbf^\der(\Rbb)^+$ the corresponding lattice space, and $\Mscr=\Omega\bsh\Yscr^+$ the S-space. Fix a base point $y\in \Yscr^+\subset Y^+$.

(1) The orbit map $\kappa_y:\Pbf^\der(\Rbb)^+\ra \Yscr^+,\ g\mapsto gy$ is a submersion with compact fibers. The isotropy subgroup $K_y$ of $y$ in $\Pbf^\der(\Rbb)^+$ is a macimal compact subgroup of $\Pbf^\der(\Rbb)^+$.

The left invariant Haar measure $\nu_\Pbf$ on $\Pbf^\der(\Rbb)^+$ passes to a left invariant measure $\mu_\Yscr=\kappa_{y*}\nu_\Pbf$ on $\Yscr^+$, which is independent of the choice of $y$.

(2) Taking quotient by congruence subgroups, the orbit map $\kappa_y:\Gamma^\dagger\bsh\Pbf^\der(\Rbb)^+\ra \Gamma\bsh \Yscr^+,\ \ \Gamma^\dagger g\mapsto\Gamma gy$ is a submersion with compact fibers. The push-forward $\kappa_{y*}$ sends $\nu_\Omega$ to a canonical probability measure $\mu_{\Mscr}$ on $\Mscr=\Gamma\bsh \Yscr^+$, independent of the choice of $y$.

(3) Let $M'\subset\Mscr$ be a special lattice subspace defined by $(\Pbf',Y';Y'^+)$, and take $y\in Y'^+\subset Y^+$. Then we have the commutative diagram $$\xymatrix{ \Omega'\ar[r]^\subset\ar[d]^{\kappa_y} & \Omega\ar[d]^{\kappa_y}\\ \Mscr'\ar[r]^\subset & \Mscr}$$ with $\Omega'$ the special lattice space associated to $M'$, $\Mscr'$ the corresponding special S-subspace. In particular we have $\kappa_{y*}\nu_{\Omega'}=\mu_{\Mscr'}$.

Similarly, for the fibration over the pure base $\pi:M\ra S=\Gamma_\Gbf\bsh X^+$ with $\Gamma=\Gamma_\Vbf\rtimes\Gamma_\Gbf$, we have the submersions $\pi:\Omega\ra\Omega_\Gbf:=\Gamma_\Gbf^\dagger\bsh\Gbf^\der(\Rbb)^+$, $\pi:\Mscr'\ra\Sscr=S$, together with $\pi_*\nu_\Omega=\nu_{\Omega_\Gbf}$ and $\pi_*\mu_\Mscr=\mu_S$.

\end{definition-proposition}

\begin{proof}
  It suffices to replace the $\Vbf$'s etc. in \cite{chen kuga} 2.17 and 2.18 by $\Wbf$'s etc. as the proof there already works for general unipotent $\Vbf$'s.
\end{proof}

In the pure case, we have the notion of $\Tbf$-special sub-object, where $\Tbf$ is  the connected center of the $\Qbb$-group defining the subdatum, the special subvarieties, etc. In the mixed case, the connected center is of the form $w\Tbf w^\inv$ following the notations in \ref{description of subdata}, and we prefer to separate $\Tbf$ and $w$, because $\Tbf$ provides information on the image in the pure base, and $w$ describes how the pure section has been translated from the given one defined by $(\Gbf,X)\subset(\Pbf,Y)$. In Introduction we have seen motivation of this notion for Kuga varieties via the description of special subvarieties as torsion subschemes in some subfamily of abelian varieties.

\begin{definition}[$(\Tbf,w)$-special subdata]\label{tw-special subdata}

Let $(\Pbf,Y)=(\Ubf,\Vbf)\rtimes(\Gbf,X)$ be a mixed Shimura datum, with $\Wbf$ the central extension of $\Vbf$ by $\Ubf$ as the unipotent radical of $\Pbf$. Take $\Tbf$ a $\Qbb$-torus in $\Gbf$ and $w$ an element in $\Wbf(\Qbb)$.

(1) A subdatum $(\Pbf',Y')$ of $(\Pbf,Y)$ is said to be $(\Tbf,w)$-special if it is of the form $(\Ubf',\Vbf')\rtimes(w\Gbf'w^\inv,wX')$ presented as in \ref{description of subdata}, with $\Tbf$ equal to the connected center of $\Gbf'$. In the language of \cite{ullmo yafaev}, $(\Gbf',X')$ is a $\Tbf$-special subdautm of $(\Gbf,X)$, and the element $w\in\Wbf(\Qbb)$ conjugates it to a pure section of $(\Pbf',Y')$ in the sense of \ref{remark on pure sections}.

(2) Similarly, if $M=\Gamma\bsh Y^+$ is a connected mixed Shimura variety  defined by $(\Pbf,Y;Y^+)$, then a special subvariety of $M$ is $(\Tbf,w)$-special if it is defined by some (connected) $(\Tbf,w)$-special subdatum.

We also define notions such as $(\Tbf,w)$-special lattice subspaces, $(\Tbf,w)$-special S-subspaces, in the evident way.
\end{definition}

\begin{remark}[minimality of Mumford-Tate groups]\label{minimality of Mumford-Tate groups}

In \ref{fibred mixed shimura data} we have imposed the condition (e) of minimality for pure Shimura data. For example, if $(\Gbf,X)$ contains some subdatum $(\Gbf',X')$, and that $\Gbf$ contains $\Qbb$-subgroup $\Gbf''\supsetneq\Gbf'$ such that $\Gbf''^\der=\Gbf'^\der$, then $(\Gbf'',X')$ is NOT a pure Shimura subdatum in our sense because $\Gbf''$ is NOT minimal. It is a subdatum in the sense of \cite{pink thesis}.

 
 For a general $\Qbb$-subtorus $\Tbf$, the set of $\Tbf$-special subdata of $(\Gbf,X)$ in our sense could be empty, simply because $\Tbf$ has been chosen to be ''too large', although one might find subdatum $(\Gbf',X')$ in usual sense such that $\Tbf$ is the connected center of $\Gbf$.

\end{remark}


We will also use the following variant to state our main results on equidistribution of special subvarieties, which is closely related to the formulation of bounded Galois orbits studied in Section 4, inspired by the pure case studied in \cite{ullmo yafaev}. Subsets of closed subvarieties of a $\Qac$-variety of finite type are always countable, hence we talk about sequences of special subvarieties indexed by $\Nbb$ instead of ''families'', ''subsets'', etc.

\begin{definition}[bounded sequence]\label{bounded sequence} Let $M=\Gamma\bsh Y^+$ be a connected mxied Shimura variety defined by $(\Pbf,Y;Y^+)$, with pure section $(\Gbf,X)$ and Levi decomposition $\Pbf=\Wbf\rtimes\Gbf$. Fix a finite set $B=\{(\Tbf_1,w_1),\cdots,(\Tbf_r,w_r)\}$ with $\Tbf_i$ $\Qbb$-torus in $\Gbf$ and $w_i\in\Wbf(\Qbb)$, $i=1,\cdots,r$.

(1) A special subvariety of $M$ is said to be bounded by $B$ (or $B$-bounded) if it is $\TW$-special for some $\TW\in B$.  A sequence $(M_n)$ of special subvarieties in $M$ is said to be bounded by $B$ if  each $M_n$ is $B$-bounded.

(2) Similarly,  a sequence of special lattice subspaces resp.   of special S-subspaces is bounded by $B$  if its members are defined by to $\TW$-special subdatum for $\TW\in B$.

(3) For $\Omega$ resp. $\Mscr$ the lattice space resp. the S-space associated to $M$ we write $\Pscr_B(\Omega)$ resp. $\Pscr_B(\Mscr)$ for the set of canonical probability measures on $\Omega$ resp. on $\Mscr$ associated to $B$-bounded special subvarieties. Note that for $\Omega'\subset\Omega$ a $B$-special lattice subspace, we identify the canonical probability measure on $\Omega'$ as a probability measure on $\Omega$ whose support equals $\Omega'$. The case of special S-subspaces is similar.

\bigskip We call $B$ a finite bounding set.
\end{definition}

! \emph{the following remark should be moved to Section 4}

\begin{remark}[independence of $\Gamma$-conjugation]\label{independence of gamma-conjugation}

(1) For the connected mixed Shimura variety $M=\Gamma\bsh Y^+$ as above, with $\Gamma=\Gamma_\Wbf\rtimes\Gamma_\Gbf$, we may conjugate $\Tbf$ by $\alpha\in\Gamma_\Gbf$, and also translate $w$ by $\Gamma_\Wbf$.

\begin{itemize}

\item In the pure case, if a special subvariety is defined by $(\Gbf',X';X'^+)$, then it is also defined by $(\alpha\Gbf'\alpha^\inv; \alpha X'; \alpha X'^+)$;

\item similarly in the mixed case, one may compute directly that translate $w$ by an element in $\Gamma_\Wbf$ gives the same special subvariety, using \ref{group law}.

\end{itemize}

Note that $w$ and $\Gamma_\Wbf$ generates a congruence subgroup of $\Wbf(\Qbb)$, in which $\Gamma_\Wbf$ is of finite index. This is reduced to the commutative case $\Wbf=\Ubf$ or $\Wbf=\Vbf$, which is evident because $\Gamma_\Wbf$ is a $\Zbb$-lattice, and coefficients of $w$ are rational \wrt $\Gamma_\Wbf$.

In some sense, in $\Gamma_\Gbf$ and $\Gamma$ is encoded information on the integral structure of the Shimura varieties, such as good or bad reduction, etc.  Our estimation of the Galois orbits involves the position of $\Tbf$

(2) The non-connected case is similar, except that it is better to talk about special subvarieties using adelic Hecke translations inside the total mixed Shimura tower $\limproj_KM_K(\Pbf,Y)$, which is a pro-scheme with a continuous action of $\Pbf(\adele)$ in the sense of \cite{deligne pspm} ?.?.  We will not need this complicated version. 

Note in the unipotent fiber, $w\in\Wbf(\adele)$ and $K_\Wbf$ generates a \cosg containing $K_\Wbf$ as a subgroup of finite index, similar to the connected case in (1).
\end{remark}

\section{bounded equidistribution of special subvarieties}

We first consider the case when the bound $B$ consists of one single element $\TW$, and we write $\Pscr_\TW(\Sscr)$ in place of $\Pscr_B(\Sscr)$ for $\Sscr\in\{\Omega,\Mscr\}$. This is exactly the analogue of the $\Tbf$-special case for pure Shimura varieties, and is also deduced from the following theorem of S. Mozes and N. Shah:

\begin{theorem}[Mozes-Shah]\label{mozes shah} Let $\Omega=\Gamma\bsh\Hbf(\Rbb)^+$ be the lattice space associated to a $\Qbb$-group of type $\Hscr$ and a congruence subgroup $\Gamma\subset\Hbf(\Rbb)^+$. Write $\Pscr_h(\Omega)$ for the set of canonical measures on $\Omega$ associated to lattice subspaces defined by $\Qbb$-subgroups of type $\Hscr$. Then $\Pscr_h(\Omega)$ is compact for the weak topology as a subset of the set of Radon measures on $\Omega$, and the property of ''support convergence'' holds on it: if $\nu_n$ is a convergent sequence in $\Pscr_h(\Omega)$ of limit $\nu$, then we have $\supp\nu_n\subset \supp\nu$ for $n\geq N$, $N$ being some positive integer, and the union $\bigcup_{n\geq N}\supp\nu_n$ is dense in $\supp\nu$ for the analytic topology.

\end{theorem}

Here the notion of lattice subspaces associated to $\Qbb$-subgroup of type $\Hscr$ is defined in the same way as we have seen in \ref{lattice spaces and canonical measures}: we have the uniformization map $\wp_\Gamma:\Hbf(\Rbb)^+\ra\Omega$; the images $\Omega'=\wp_\Gamma\bsh\Hbf'(\Rbb)^+$ associated to $\Qbb$-subgroups $\Hbf'$ of type $\Hscr$ are called lattice subspaces, and they carry canonical probability measures induced by the Haar measures on $\Hbf'(\Rbb)^+$, which we view as probability measures on $\Omega$ of support equal to $\Omega'$.

Similar to the rigid $\Tbf$-special subdata  in \cite{chen kuga} 3.1, we have the following recovering property for $\TW$-special sub-objects:

\begin{lemma}[recovery, cf. \cite{chen kuga} 3.5]\label{recovery}

Let $\Omega=\Gamma^\dagger\bsh\Pbf^\der(\Rbb)^+$ be the lattice space associated to a connected mixed Shimura variety $M$ defined by $(\Pbf,Y;Y^+,\Gamma)$. Let $\Ccal$ be the set of $\TW$-special ''$\Qbb$-subgroups'' of $\Pbf$, i.e. $\Qbb$-subgroups $\Pbf'$ that come from $\TW$-special subdata $(\Pbf',Y')$ of $(\Pbf,Y)$. Then the map $$\Ccal\ra\{\mathrm{lattice\ subspaces\ of\ }\Omega\},\ \Pbf'\ra\wp_\Gamma(\Pbf'^\der(\Rbb)^+)$$ is injective. 
\end{lemma}
\begin{proof}
  If $(\Pbf_1,Y_1)$ and $(\Pbf_2,Y_2)$ are $\TW$-special and give the same lattice space $\Omega'$, then computing the tangent space of $\Omega'$ shows that $\Pbf_1^\der=\Pbf_2^\der$. We may write $\Pbf_i=\Wbf_i\rtimes(w\Gbf_iw^\inv)$, $\Gbf_i$ being $\Qbb$-subgroups of $\Gbf$ of connected center $\Tbf$, then $\Wbf_1=\Wbf_2$ as the unipotent radical of $\Pbf_1^\der=\Pbf_2^\der$, and $\Gbf_1^\der=\Gbf_2^\der$ as their image in $\Gbf$, hence $\Gbf_1=\Gbf_2$ as they are of the same connected center $\Tbf$.
\end{proof} 

We also have \begin{lemma}

The set of maximal $\TW$-special subdata in $(\Pbf,Y)$ is finite.

\end{lemma}

Its proof is the same as in \cite{chen kuga} 3.7.

\begin{theorem}[equidistribution of $\TW$-special subspaces]\label{equidistribution of tw-special subspaces} Let $\Sscr$  be the lattice space (resp. the S-space) associated to a connected mixed Shimura variety $M=\Gamma\bsh Y^+$ defined by $(\Pbf,Y;Y^+)=(\Ubf,\Vbf)\rtimes(\Gbf,X)$. Fix a pair $(\Tbf,w)$ as in \ref{tw-special subdata}. The set $\Pscr_B(\Sscr)$ is compact for the weak topology, and the property of support convergence holds in it in the sense of \ref{mozes shah}.

\end{theorem}

\begin{proof}[proof for  lattice subspaces] The proof is completely parallel to the Kuga case treated in \cite{chen kuga} Section 4, so we only outline the bounded mixed case.

 We have the lattice space $\Omega=\Gamma^\dagger\bsh\Pbf^\der(\Rbb)^+$. From \ref{commutator} we see that  $\Pscr_\TW(\Omega)$ is a subset of $\Pscr_h(\Omega)$, and it suffices to show that $\Pscr_\TW(\Omega)$ is closed in $\Pscr_h(\Omega)$ for the weak topology.

We thus take a convergent sequence $\nu_n$ in $\Pscr_h(\Omega)$ of limit $\nu$, such that $\nu_n\in\Pscr_\TW(\Omega)$ for all $n$, and we assume for simplicity that $\supp\nu_n\subset\supp\nu$ for all $n$. The $\nu_n$'s are associated to $\TW$-special subdata $(\Pbf_n,Y_n)=\Wbf_n\rtimes(w\Gbf_nw^\inv, wX_n)$. Their supports are $\Omega_n=\wp_\Gamma\bsh\Pbf_n^\der(\Rbb)^+$, with $\Pbf_n^\der=\Wbf_n\rtimes w\Gbf_n^\der w^\inv$.

The limit $\nu$ is associated to some $\Qbb$-subgroup $\Pbf_\nu$ of type $\Hscr$, written as $\Pbf_\nu$ with unipotent radical $\Wbf_\nu$. 

The commutativity of $\Vbf$ in \cite{chen kuga} 4.1 - 4.5 is not used, and the arguments work for general unipotent radical of mixed Shimura data.

Hence the union $\bigcup_n\Pbf^\der_n$ generates a $\Qbb$-subgroup of type $\Hscr$, of the form $\Wbf'\rtimes\Hbf'$, whose associated lattice subspace supports $\nu$. We then have

\begin{itemize}
 \item  $\Wbf'$ is actually a unipotent $\Qbb$-subgroup of $\Wbf$ generated by $\bigcup_n\Wbf_n$; it is a central extension of $\Vbf'$ by $\Ubf'$ for $\Ubf'$ generated by $\bigcup_n\Ubf_n$ and $\Vbf'$ by $\bigcup_n\Vbf_n$; $\Ubf'$, $\Vbf'$, and $\Wbf'$ are stable under $w\Tbf w^\inv$.
 
 \item $\Hbf'$ is a connected semi-simple $\Qbb$-group, whose image $\pi(\Hbf')$ in $\Gbf$ under the natural projection is generated by $\bigcup_n\Gbf_n^\der$; in particular the image is centralized by $\Tbf$.
 \item  $\Wbf'\rtimes\Hbf'$ contains the Zariski closure of $\bigcup_n w\Gbf_n^\der w^\inv$, which is $w\pi(\Hbf') w^\inv$, a connected semi-simple $\Qbb$-group centralized by $w\Tbf w^\inv$.
     
  \item $\Pbf':=\Wbf'\rtimes w\Tbf\pi(\Hbf')w^\inv$ defines some $\TW$-special subdatum $(\Pbf',Y')$, whose associated $\TW$-special canonical measure on $\Omega$ equals $\nu=\lim_n\nu_n$.
\end{itemize}
Hence $\Pscr_\TW(\Omega)$ is closed.
\end{proof}

\begin{proof}[proof for S-subspaces]
Again the commutativity of $\Vbf$ is not used in \cite{chen kuga} Section 5, and we merely outline the main arguments. Note that we have assumed \ref{level of product type}, so $\Mscr=\Gamma\bsh \Yscr^+\ra S=\Gamma_\Gbf\bsh X^+$ is a fibration, whose fibers are torsors by compact real tori over complex abelian varieties.

\begin{itemize}
  
  \item There exists a compact subset $K\TW$ of $\Yscr^+$ such that if $\Mscr'\subset |Mscr$ is a $\TW$-special S-subspace, then $\Mscr=\wp_\Gamma(\Yscr'^+)$ is given by some connected $\TW$-special subdatum $(\Pbf',Y';Y'^+)$, with real part $\Yscr'^+$ meeting $K\TW$ non-trivially.
  
  \item The set $\Pscr_\TW(\Mscr)$ is compact for the weak topology: if $\mu_n$ is a sequence of $\TW$-special canonical measures on $\Mscr$ defined by $(\Pbf_n,Y_n;Y_n^+)$, given as $\mu_n=\kappa_{y_n*}\nu_n$ for $y_n\in K\TW$ and $\nu_n$ the canonical measure associated to $\Omega_n=\Gamma_n^\dagger\bsh \Pbf_n^\der(\Rbb)^+$, then up to restriction to subsequences, we may assume that $y_n$ converges to some $y\in K\TW$ and $(\nu_n)$ converges to some $\nu$ associated to a $\TW$-special subdatum $(\Pbf',Y;Y'^+)$ with $y\in\Yscr'^+\cap K\TW$. Thus $\mu_n$ congerges to $\mu=\kappa_{y*}\nu$.

\end{itemize}

The property of support convergence holds similarly.
\end{proof}

\begin{corollary}[bounded equidistribition]\label{bounded equidistribution} (1) For $B$ a finite bounding set, $\Sscr\in\{\Omega,\Mscr\}$, we have $\Pscr_B(\Sscr)=\coprod_{\TW\in B}\Pscr_\TW(\Sscr)$. In particular, $\Pscr_B(\Sscr)$ is compact for the weak topology, and the support convergence holds in it.

(2) For $\Sscr=\Omega$ (resp. $\Sscr=\Mscr$), the closure of a sequence of special lattice subspaces (resp. of special S-subspaces) bounded by $B$ for the analytic topology is a finite union of special lattice subspaces (resp. of special S-subspaces) bounded by $B$.

\end{corollary}

\begin{proof}
  
  (1) This is clear because $\Pscr_B(\Sscr)$ is a finite union of compact subsets of the set of Radon measures on  $\Sscr$. The property of support convergence holds because if a sequence $(\mu_n)$ converges to $\mu$, then it contains a subsequence that converges into $\Pscr_\TW(\Sscr)$  for some $\TW\in B$. Hence the $\nu\in\Pscr_\TW(\Sscr)$. All the convergent subsequence of $(\mu_n)$ are of the same limit, so it is not possible to have an infinite subsequence lying outside $\Pscr_\TW(\Omega)$. Hence the sequence itself is in $\Pscr_\TW(\Sscr)$.
  
  (2) This is clear using the convergence of measures and the property of support convergence.
\end{proof}

\begin{corollary}[bounded Andr\'e-Oort]\label{bounded Andr\'e-Oort} Let $M$ be a connected Shimura variety defined by $(\Pbf,Y)=(\Ubf,\Vbf)\rtimes(\Gbf,X)$, with $B$ a finite bounding set. Let $(M_n)$ be a sequence of special subvarieties bounded by $B$. Then the Zariski closure of $\bigcup_nM_n$ is a finite union of special subvarieties bounded by $B$.

\end{corollary}

\begin{proof}
  This is clear because analytic closure is finer than Zariski closure, and that the S-subspaces are Zariski dense in the corresponding special subvarieties.
\end{proof}

\begin{remark}
  In \cite{ratazzi ullmo}, Ratazzi and Ullmo has shown, only using tools from harmonic analysis, that the analytic  closure of a sequence of compact subtori in a given compact torus $\Rbb^d/\Zbb^d$ remains a compact tori. Our proof of bounded equidistribution can be viewed as the analogue of this result in the setting of special subvarieties of mixed Shimura subvarieties, using ergodic-theoretic tools only.
\end{remark}

\end{document}